\documentclass{article}
\author{Ganzhinov Mikhail, \"Ostergård Patric R. J.\\
Department of Information and Communications Engineering\\
Aalto University School of Electrical Engineering\\
P.O.\ Box 15400, 00076 Aalto, Finland}
\date{}
\usepackage{graphicx}
\usepackage{float}
\usepackage[margin=1.0in]{geometry}
 \usepackage{amsfonts}
\usepackage{amsmath}
\usepackage{amsthm}
\usepackage{hyperref}
\usepackage{tikz}
\usepackage{color}

\newcommand{\Sh}{\textrm{S}}
\newcommand{\Aut}{\textrm{Aut}}
\newcommand{\SL}{\textrm{SL}}
\newtheorem{proposition}{Proposition}

\title{Spherical codes with prescribed signed permutation automorphisms inside shells of low-dimensional integer lattices}
\begin{document}
\maketitle
\begin{abstract}
Let $\Sh(n,t,k)$ be the maximum size of a code containing only vectors of the $k$th shell of the integer lattice $\mathbb{Z}^n$ such that the inner product between distinct vectors does not exceed $t$. In this paper we compute lower bounds for $\Sh(n,t,k)$ for small values of $n$, $t$ and $k$ by carrying out computer searches for codes with prescribed automorphisms. We prescribe groups of signed permutation automorphisms acting transitively on the pairs of coordinates and coordinate values as well as other closely related groups of automorphisms. Several of the constructed codes lead to improved lower bounds for spherical codes.
\end{abstract}
\bigskip
\section{Introduction}

A \emph{spherical code} $\mathcal{C}$ is a finite set of points on the unit sphere $S^{n-1}$ in $\mathbb{R}^n$. Such a code $\mathcal{C}$ has parameters $(n,N,t)$ if there are $N$ points in total and the inner product between vectors from the origin to any two distinct points is at most $t$. The \emph{angle set} of the spherical code $\mathcal{C}$ is defined as $\mathcal{A}_\mathcal{C}=\{\langle x,y\rangle: x,y\in\mathcal{C},x\neq y\}$ while individual elements in $A_\mathcal{C}$ are called \emph{angles}.  Two spherical codes are said to be \emph{equivalent} if one can be obtained from the other through an isometry. An isometry (orthogonal transform) mapping a spherical code onto itself is called an automorphism of the code. The set of all automorphisms of a spherical code $\mathcal{C}$ is a group under composition, called the \emph{automorphism group} of the code $\mathcal{C}$ and denoted by $\Aut(\mathcal{C})$. A subgroup of the automorphism group is called a \emph{group of automorphisms}.

The problem of minimizing $t$ for given $n$ and $N$ is known as the \emph{Tammes problem} \cite{Tammes} for the $(n-1)$-sphere. Spherical codes attaining the solution of the Tammes problem for some $n$ and $N$ are called \emph{optimal}. Instead of minimizing $t$ for given $n$ and $N$ one may be interested in maximizing $N$ for given $n$ and $t$. For most parameters exact solutions to the Tammes problem are not known and only bounds for $t$ are available. The lower bounds for $t$ are usually obtained by theoretical arguments \cite{BV,KL}, while upper bounds often come from explicit constructions. 

For small values of $n$ and $N$ various optimization algorithms have been used to obtain upper bounds for $t$ \cite{Cohn,Sl,SHS,W}. Upper bounds have also been obtained by specific constructions utilizing either algebraic or combinatorial arguments \cite{DF,IJM}. In such cases vectors of the constructed spherical codes often have clear structure, for example, all the vectors can be found inside a shell of an (properly scaled) integer lattice. The most notable examples of this are the root system of an $E_8$ lattice and the inner shell of the Leech lattice producing universally optimal spherical codes \cite{Unopt}, which can be embedded into the 4:th and 24:th shells of the integer lattices $\mathbb{Z}^8$ and $\mathbb{Z}^{24}$, respectively \cite{SPLAG}. While these examples are exceptional objects, scaled shells of integer lattices can also generally contain good spherical codes. We denote by $\Sh(n,t,k)$ the maximum size of a code inside the $k$th shell of the integer lattice $\mathbb{Z}^n$ such that the inner product between distinct vectors is at most $t$.

In this work we carry out computer searches for codes containing only vectors of a single shell of an integer lattice in order to obtain lower bounds for $\Sh(n,t,k)$, $6\leq n\leq 14$. A common technique \cite{BEÖVW,LÖB,LÖT} used to limit the search is to prescribe group of automorphisms of a code. We prescribe groups of automorphisms which are inherited from the parent lattice $\mathbb{Z}^n$, mainly the groups of signed permutation automorphisms acting transitively on the pairs of coordinates and coordinate values as well as other closely related groups. Several of found codes improve lower bounds for spherical codes.

We discuss our approach in more detail in Sections \ref{sec:Int} and \ref{sec:Gr}. In Section \ref{sec:Tab} we tabulate the computed lower bounds for $\Sh(n,k,t)$ and list parameters of best spherical codes obtained based on our computations.

\section{Integer lattices}\label{sec:Int}

A \emph{lattice} is a discrete additive subgroup of $\mathbb{R}^n$. A lattice $\Lambda$ generated by a set of linearly independent vectors $B=\{v_1,v_2,\ldots,v_k\}\subset\mathbb{R}^n$ is the set of all the integer linear combinations,
\begin{displaymath}
\Lambda=\left\{\sum_{i=1}^{k}a_iv_i:a_i\in\mathbb{Z}\right\}.
\end{displaymath}
We say that $\Lambda$ is a lattice of \emph{dimension} $n$. The  automorphism group of a lattice in $\mathbb{R}^n$ is the subgroup of the orthogonal group $\SL(n)$ that preserves the lattice. A \emph{shell} $s_m$ of a lattice $\Lambda$ is the set of lattice vectors of squared norm $m$, that is, $s_m=\{v\in\Lambda:\| v\|^2=m\}$.

An \emph{integer lattice} in $\mathbb{R}^n$ is a lattice generated by the vectors of the standard basis obtained from the vector $v=(1,0,\ldots,0)\in\mathbb{R}^n$ by permuting its coordinates. This lattice contains the vectors of $\mathbb{R}^n$ with integer coordinates and is usually denoted by $\mathbb{Z}^n$. The automorphism group of $\mathbb{Z}^n$, denoted by $\Aut(\mathbb{Z}^n)$, is generated by all permutations and sign changes of coordinates and is of order $2^nn!$. The group $\Aut(\mathbb{Z}^n)$ is sometimes called the \emph{signed permutation group} and its elements \emph{signed permutations}. This automorphism group acts transitively on the shells $s_1$, $s_2$ and $s_3$, but not on other shells. For example, orbit representatives of the shell $s_9$ of the $10$-dimensional integer lattice are
\begin{displaymath}
(1,1,1,1,1,1,1,1,1,0)\textrm{, }(2,1,1,1,1,1,0,0,0,0)\textrm{, }(2,2,1,0,0,0,0,0,0,0)\textrm{ and }(3,0,0,0,0,0,0,0,0,0).
\end{displaymath}
One can notice that the last orbit of the shell $s_9$ is a scaled copy of the shell $s_1$ of the same lattice. This leads to following more general observation.

\begin{proposition}\label{prop:1}
Linear maps defined by integer matrices with orthogonal rows of squared norm $k\in\mathbb{N}$ embed the shell $s_m$ into the shell $s_{km}$ for all $m\in\mathbb{N}$. This embedding preserves angles between vectors.
\end{proposition}
\noindent By slightly abusing terminology, we call such matrices \emph{integer orthogonal matrices}.

The shell $s_m$ of any integer lattice is at most $2m$-angular, with inner products between distinct shell vectors belonging to the set $\{m-1,m-2,\ldots,-m\}$. Shells of integer lattices have the following well known property.

\begin{proposition}
Let $\mathcal{C}\subset\mathbb{R}^{n}$ be a spherical code with rational angle set spanning the entire space $\mathbb{R}^n$. Then there exists a shell $s_m$ of an integer lattice $\mathbb{Z}^{q}$, $q\geq n$ such that $\mathcal{C}$ can be embedded into $s_m$ via isometries and scalings.
\end{proposition}
\begin{proof}
First choose a basis inside the spherical code $\mathcal{C}$ and apply the Gram-Schmidt orthogonalization process to it. Multiply every vector of the newly formed orthogonal basis by a natural number in such a way that
\begin{enumerate}
\item all (Euclidean) norms of basis vectors are square roots of natural numbers, and
\item all the inner products between the basis vectors and the vectors of the code $\mathcal{C}$ are integers.
\end{enumerate}
Let $v_1,v_2,\ldots,v_n$ be the vectors of this orthogonal basis. Each vector $v_j$ of norm $\sqrt{m_j}$ can be embedded into a space $V_j\cong\mathbb{R}^{m_j}$ as an integral vector $(1,1,\ldots,1)$. Now, a scaled isomorphic copy of the code $\mathcal{C}$ can be constructed in the space $V=\oplus_{1\leq j\leq n}V_j$ in a natural way. The space $V$ has the standard basis induced by standard bases of the spaces $V_j$. All vectors of the constructed copy have rational inner products with the standard basis vectors of $V$ and thus become integral after proper scaling.
\end{proof}
Often, the dimension of the space $V$ in the proof can be lowered considerably and the described embedding is far from optimal from that point of view.

\section{Code construction}\label{sec:Gr}

Codes embedded into a shell of an integer lattice with large angular separation between the codevectors can inherit some symmetries of the lattice. The search of such codes is straightforward as long as signed permutation automorphisms form a group of suitable size.

To study codes with prescribed signed permutation automorphisms in $V=\mathbb{R}^n$, it is convenient first to map any vector $v\in\mathbb{R}^n$ into $W=\mathbb{R}^{2n}$ by the following angle preserving linear map:
\begin{displaymath}
f:V\rightarrow W\textrm{, }f((a_1,\ldots,a_n))=(a_1,-a_1,\ldots,a_n,-a_n)\textrm{ for all }(a_1,\ldots,a_n)\in\mathbb{R}^n.
\end{displaymath}
The group of signed permutations acting on integer vectors of $V$ can be studied via its isomorphic permutation representation on $f(V)\subset W$ acting on $f$-images of integer vectors of $V$. The permutation representation in question permutes the standard basis vectors of $W$ containing all permutations stabilizing the partition of coordinates
\begin{displaymath}
\{1,2\},\{3,4\},\ldots,\{2n-1,2n\}.
\end{displaymath}
A \emph{block system} for the action of a group $G$ on a set $S$ is a partition of $S$ that is $G$-invariant. Permutation groups with block systems with blocks of size 2 can represent signed groups according to the presentation above.

The problem of finding the largest code for a given group of automorphisms $G\leq\Aut(\mathbb{Z}^n)$, an integer lattice shell $s_k$ and inner product at most $t$ can be transformed into an instance of a weighted maximum clique problem. A weighted graph $\Gamma$ is constructed where each $G$-orbit with the inner product between any two vectors is at most $t$ is represented by a vertex and the weight of a vertex is the cardinality of the corresponding orbit. There is a edge between two vertices of the graph if the inner product between any two vectors of the orbits is at most $t$. Now, the largest codes with prescribed group of automorphism $G$ correspond to maximum weight cliques of $\Gamma$.

In practice it is necessary to restrict the types of prescribed signed permutation groups. In this work we consider
\begin{enumerate}
\item signed permutation groups acting transitively on all the coordinate/sign pairs or leaving a few coordinates and their values unchanged,
\item subgroups of signed permutation groups acting transitively on all the coordinate/sign pairs stabilizing coordinates, blocks of coordinates or coordinate/sign pairs,
\item permutation groups acting transitively on all coordinates or leaving a few coordinates unchanged, and
\item some other small permutation and signed permutation groups.
\end{enumerate}
Transitive permutation groups were classified \cite{CH,HRT,H} up to degree 48, meaning that all signed permutation groups acting transitively on coordinate/sign pairs are essentially known up to degree 24. We systematically consider signed permutation groups up to degree 14, acting transitively on all coordinate/sign pairs or fixing a coordinate and its sign. Other types of groups of automorphisms we prescribe less systematically. We use GAP \cite{G} and its GRAPE package \cite{Grape} in our calculations. Some groups are excluded from the search if the associated weighted graph $\Gamma$ is too large or the problem too time consuming. Finally, in some cases we use $Y_n$ constructions of \cite{EZ} (see also \cite{ERS}) and other methods if they give better results.

\section{Results}\label{sec:Tab} 

In this section we report the computed lower bounds for $\Sh(n,t,k)$. For $6\leq n\leq 11$ and $12\leq n\leq14$ we study shells $s_k$ within the ranges $4\leq k\leq 15$ and $4\leq k\leq 26-n$ respectively. For $n\leq5$ it may be possible to determine in most cases the exact values of $\Sh(n,t,k)$ for $4\leq k\leq 15$, hence corresponding bounds were not computed here.

The lower bounds are tabulated in Tables \ref{tab:t2} to \ref{tab:t10}. A construction type of a code leading to a bound is given after the bound, many lower bounds can be obtained via multiple construction types; abbreviations for the types are listed in Table \ref{T1}. Some additional lower bounds for parameters not included into previous tables are given in Table \ref{tab:t11}; these bounds were obtained through non-systematic search but the bounds themselves seem to be strong enough to be mentioned here. Explicit codes attaining the lower bounds are given in \cite{Gan}.

\begin{table}[H]  
\centering 
\begin{tabular}{cl}\hline\hline   
 Abbreviation &Explanation\\\hline     
sq&  The code can be constructed by prescribing signed permutation automorphisms acting\\
    & transitively on $2(d-q)$ pairs of coordinates and signs. The remaining pairs of \\
    & coordinates and signs remain unchanged. If $q=0$ then $q$ omitted.\\\hline
pq & The code can be constructed by prescribing automorphisms permuting transitively on\\
    &  $(d-q)$ coordinates and fixing all the remaining coordinates. If $q=0$ then $q$ omitted.\\\hline
n  & The code can be constructed by prescribing automorphism groups permuting\\
    & coordinates or pairs of coordinates and signs intransitively.\\\hline
oq& The code can be constructed inside the shell $s_q$ which is embedded into the shell $s_k$\\
    & by Proposition \ref{prop:1}.\\\hline
x  & The code can be constructed by one of the $Y_n$ constructions of \cite{EZ}. \\\hline
+ & The code contains additional, explicitly added vectors. \\\hline
-  & The lower bound is not computed. \\\hline
\end{tabular}
\caption{Abbreviations for types of constructions used} 
\label{T1}
\end{table}
\begin{table}[H]
\small\begin{tabular}{|c||l|l|l|l|l|l|l|l|l|l|l|l|}\hline
$t\backslash k$ & 4        & 5        & 6      &7       &8        &9        &10    &11     &12    &13      &14        &15\\\hline\hline
1                      &18s2   &18s1    &12s    &12s    &12s     &12s     &12s   &12s   &12s  &12s     &12s      &12s\\\hline
2                      & 60s     &32s1   &32s     &18p   &18s     &18n   &18o5 &13s2+ &15n  &12s    &12s      &12s\\\hline
3                      &           &56n    &36s     &32n   &26p1+&24p    &21n   &19n   &21n  &16n   &14n     &14n\\\hline
4                      &           &          &160o3 &66n+&60o4  &48s     &32o5 &32s2  &32s  &20n    &20n     &18n+\\\hline
5                      &           &          &          &192s &100s2 &70n    &60s   &40s1  &32s  &32n    &28n+   &32s\\\hline
6                      &           &          &          &       &252s   &116n   &72n   &58n+ &46s  &40n   &36s     &32s\\\hline
7                      &           &          &          &       &          &196n   &144s &96s   &66n  &50s1   &48n    &42p\\\hline
8                      &           &          &          &       &          &          &504s &180s &164n &96s    &66o7   &55p\\\hline
\end{tabular}
\caption{Lower bounds for $\Sh(6,t,k)$}
\label{tab:t2}
\end{table}
\begin{table}[H]
\small\begin{tabular}{|c||l|l|l|l|l|l|l|l|l|l|l|l|}\hline
$t\backslash k$ & 4         & 5         & 6         &7      &8           &9         &10      &11      &12       &13        &14      &15\\\hline\hline
1                      &  24n    &24s1     &16n     &16s    &14s       &14s      &14s      &14s    &14s      &14s      &14s     &14s\\\hline
2                      &  126s   &46s3    &36n       &22p2 &28n     &20s3      &20n     &18n    &16n     &15n      &15s3+ &14s\\\hline
3                      &           &124s1+ &82s3+   &72s   &36n      &38n      &30n     &30n    &23n     &20n       &18n     &20n\\\hline
4                      &           &            &232s1   &96s3  &84s     &57s3     &48n     &34s3+&56s      &30s3+   &26n    &24n\\\hline
5                      &           &            &           &576s  &156s3  &112n    &72s1     &70n    &56s      &47s+    &37s3    &38s2\\\hline
6                      &           &            &           &        &756s    &250n+  &152n    &86s1+ &86n     &60s1     &56s3+ &48s3\\\hline
7                      &           &            &           &        &           &630s   &270n+  &180n    &106s1+&84n      &72s3    &64n\\\hline
8                      &           &            &           &        &           &          &828s+   &354n   &328n    &168s1   &130n+ &80n\\\hline
\end{tabular}
\caption{Lower bounds for $\Sh(7,t,k)$}
\label{tab:t3}
\end{table}
\begin{table}[H]
\small\begin{tabular}{|c||l|l|l|l|l|l|l|l|l|l|l|l|}\hline
$t\backslash k$ & 4         & 5         & 6         &7         &8         &9         &10      &11          &12        &13         &14       &15\\\hline\hline
1                      & 26n     & 28s1    & 17n+   &18n     &16s      &16s     &16s      &16s        &16s       &16s       &16s      &16s\\\hline
2                      & 240s    & 72n     & 64s      &32s     &36n      &24n     &28o5    &22n        &20n      &18n       &18o7    &16s\\\hline
3                      &           & 288s    & 112s    &80n     &52n      &64s      &37p1+  &40n       &28n      &26s2      &22n     &28o5\\\hline
4                      &           &            & 448s   &196n    &240s     &80s      &72s     &48n       &64s      &40n        &34n    &28o5\\\hline
5                      &           &            &           &1024s   &280n+ &170s1   &112s    &88s1      &64s      &70s1      &58n    &64s\\\hline
6                      &           &            &           &           &2160s   &460n    &288o5  &148s+    &240s    &84s1      &84n    &72o5\\\hline
7                      &           &            &           &           &           &1776s   &504s     &336n+   &240s    &146n+   &112s   &90n+\\\hline
8                      &           &            &           &           &           &           &2400s   &704s      &512s     &296n      &212s+&144n\\\hline
\end{tabular}
\caption{Lower bounds for $\Sh(8,t,k)$}
\label{tab:t4}
\end{table}
\begin{table}[H]
\small\begin{tabular}{|c||l|l|l|l|l|l|l|l|l|l|l|l|}\hline
$t\backslash k$ & 4        & 5          & 6          &7           &8           &9          &10          &11       &12        &13      &14        &15\\\hline\hline
1                      & 36s     & 32s1     &32s1    &24n       &21p1       &20x      &18s1+     &18s      &18s       &18p2  &18s       &18p2\\\hline
2                      & 306s   & 90n+    & 68n     &36s        &36n        &32s1     &30n        &24n      &24n      &20n    &21n      &20n\\\hline
3                      &           & 612s     &176n+  &128s1     &64n       &80n      &42n        &52n+    &39n      &32s1   &27p      &26p3 \\\hline
4                      &           &            &1056s   &352s1     &272s1    &108n     &88s1       &74s+    &96s       &48n   &48n      &36p2+\\\hline
5                      &           &            &           &1664s1   &464s1     &482s1   &160s1     &112n     &96s      &80s1   &74s+    &64s1\\\hline
6                      &           &            &           &             &4432s1   &706s1+ &432s       &240s1+ &244s1+ &108s  &96s1     &84s \\\hline
7                      &           &            &           &             &             &5426s    &1056s1+&580s1+  &288s    &304s1 &148n+  &152n\\\hline
\end{tabular}
\caption{Lower bounds for $\Sh(9,t,k)$}
\label{tab:t5}
\end{table}
\begin{table}[H]
\small\begin{tabular}{|c||l|l|l|l|l|l|l|l|l|l|l|l|}\hline
$t\backslash k$ & 4        & 5            & 6         &7             &8          &9          &10       &11       &12           &13       &14       &15\\\hline\hline
1                      & 44n     & 40n       &40s       &24s1+      &22n      &22p1     &21s1    &20s     &20s          &20s     &20s      &20s\\\hline
2                      & 500s   & 108s1     &104s2   &64s2        &48n      &40s       &40s      &32s2    &40s         &34s2   &24o7    &22s2+\\\hline
3                      &           & 1192x+  &252s1   &136n       &100s     &82n       &50n+   &51p3    &41s1+     &36s2+ &35p3+ &32s2\\\hline
4                      &           &              &2240s   &512n       &500o4   &148n     &122p    &96s2    &104x       &80s     &64s2    &48s2\\\hline
5                      &           &              &           &3584s+   &738s1    &546s2    &270s    &160s2   &114s1+  &114s1  &80n     &104n\\\hline
6                      &           &              &           &              &7220s    &1176n+ &1192s  &324s2+&272s2    &152n+ &160s     &104n\\\hline
7                      &           &              &           &              &            &12412s+&1952n+&904s1+&420s2+  &372s2  &204n+ &160s2\\\hline
\end{tabular}
\caption{Lower bounds for $\Sh(10,t,k)$}
\label{tab:t6}
\end{table}
\begin{table}[H]
\small\begin{tabular}{|c||l|l|l|l|l|l|l|l|l|l|l|l|}\hline
$t\backslash k$ & 4      &5          & 6       &7         &8            &9           &10        &11          &12       &13         &14       &15\\\hline\hline
1                      & 55p  &42n+    &40s1   &26s2+  &24s2+    &24u        &24p1     &24s1      &22s      &22s       &22s      &22s\\\hline
2                      & 582x&176p1   &116n   &80s1    &55n        &50n+      &42s1+  &38s2+    &40s1    &36s3+   &29n+    &22s\\\hline
3                      &        &2156x+ &352s   &216n    &144n      &132s       &83s3     &72x       &55p      &48s1+   &37s3    &40n\\\hline
4                      &        &            &4664s &704n    &540s1     &217n+    &144n    &112s3     &136x    &88s      &72n      &50s3+\\\hline
5                      &        &            &         &6736x+ &1408n+  &982s1     &328s1+ &264n+  &144s3   &152n    &132p     &104s3+\\\hline
6                      &        &            &         &           &13572s1+&2096s2+&1388s1+&448n+  &372s3+ &220s     &204n+ &144s2\\\hline
7                      &        &            &         &           &              &-            &3850p+ &2406n+ &680n+ &492n    &304n+  &284n+\\\hline
\end{tabular}
\caption{Lower bounds for $\Sh(11,t,k)$}
\label{tab:t7}
\end{table}
\begin{table}[H]
\small\begin{tabular}{|c||l|l|l|l|l|l|l|l|l|l|l|}\hline
$t\backslash k$ & 4         & 5          & 6          &7           &8           &9          &10        &11          &12       &13       &14 \\\hline\hline
1                      & 72x      &48x      & 44s2+   &32n        &27n       &28n      &26s1+   &24s         &24s     &24s     &24s\\\hline
2                      & 840s    &216s     &160o3    &84s2+    &72o4      &55s4     &48s       &40s2      &44s2+  &38s4+ &32o7\\\hline
3                      &            &2656x+ &704s     &288n      &154n+    &160o3   &92s1+   &80n        &72o4   &55s2+ &42s2+\\\hline
4                      &            &            &8928s   &1104s     &840s       &408s     &288s     &144n      &160o3  &98n    &96s\\\hline
5                      &            &            &            &13504x+&2304s     &1062s2 &528s      &276n+    &192s   &200n   &156p1+\\\hline
6                      &            &            &            &             &26760s+ &4688s   &2656o5  &768s      &840s    &320n+&288o7\\\hline
7                      &            &            &            &             &             &-          &8658p1+&2782s2+&1032s  &696n+ &504s\\\hline
\end{tabular}
\caption{Lower bounds for $\Sh(12,t,k)$}
\label{tab:t8}
\end{table}
\begin{table}[H]
\small\begin{tabular}{|c||l|l|l|l|l|l|l|l|l|l|}\hline
$t\backslash k$  & 4         & 5        & 6          &7          &8            &9           &10          &11         &12       &13\\\hline\hline
1                      & 104s    & 56n     &48s1      &48s1     &30n        &32n        &27p1+    &32n       &28s+   & 32x\\\hline
2                      & 1066x  & 288s1  &188n+   &104s1+  &72o4      &60s5      &56n        &48s1      &46s3+ &40s5+\\\hline
3                      &            &3988x+&832s      &472n     &184s1+   &170n      &99s2+    &112n+   &80s+   &64x\\\hline
4                      &            &           &11248x+&1904s1+&998n      &460s1+  &336s1     &168s1    &208s    &136n\\\hline
5                      &            &           &             &25392x+&3744s1  &1705s1   &632n      &576s1    &256s1  &312s1\\\hline
6                      &            &           &             &             &-           & -          &2658o5+&1040s1+&864s1  &344n+\\\hline
\end{tabular}
\caption{Lower bounds for $\Sh(13,t,k)$}
\label{tab:t9}
\end{table}
\begin{table}[H]
\small\begin{tabular}{|c||l|l|l|l|l|l|l|l|l|l|l|}\hline
$t\backslash k$ & 4       & 5         & 6          &7          &8            &9           &10        &11       &12\\\hline\hline
1                      & 112s  & 64x     &64n       &64s       &32n         &36n+     &32n      &32n     &28s\\\hline
2                      & 1484s& 448x   &240s2    &128s      &112s       &64n        &64n      &56n     &64o6\\\hline
3                      &         & 5464x+&1344s   &600s2+  &210s2+   &220n      &112s     &124n   &83n+\\\hline
4                      &         &           &18464x+&2912s2+&1932s     &496s2+   &448s     &224s   &240o6\\\hline
5                      &         &           &             &47424x+&6736n+  &2330s2   &954n+  &896s    &320n\\\hline
6                      &         &           &             &            &-             &-            &5464o5 &1792s  &1344s\\\hline
\end{tabular}
\caption{Lower bounds for $\Sh(14,t,k)$}
\label{tab:t10}
\end{table}
\begin{table}[H]\centering
\small\begin{tabular}{c|c|c|c}\hline\hline
$\Sh(5,10,8)\geq180$    & $\Sh(12,14,8)\geq1584$      &$\Sh(13,13,7)\geq1318$&$\Sh(15,9,5)\geq4480$\\
$\Sh(6,12,16)\geq264$  & $\Sh(12,15,7)\geq384$        &$\Sh(14,13,5)\geq362$ &$\Sh(15,55,7)\geq48$\\
$\Sh(9,17,8)\geq144$    & $\Sh(12,17,8)\geq384$        &$\Sh(14,21,8)\geq364$ &$\Sh(16,9,5)\geq5762$\\
$\Sh(10,17,8)\geq180$  & $\Sh(12,19,8)\geq264$        &$\Sh(15,8,4)\geq2564$ &\\\hline\hline
\end{tabular}
\caption{Additional bounds}
\label{tab:t11}
\end{table}
Most codes associated with the lower bounds for $\Sh(n,t,k)$ produce decent but not exceptionally strong spherical codes with some exceptions. By searching through Tables \ref{tab:t2} to \ref{tab:t10} we managed to find several spherical codes which seem to either outperform or fit nicely between the best known spherical codes with comparable parameters. The found examples are given in Table \ref{tab:t12} while the explicit codes can be found in \cite{Gan}. The new codes are compared to examples of the Appendix D of \cite{EZ}. In the Table \ref{tab:t12} by $(n,k,t,N)$ in the column structure we denote a code containing $N$ vectors of the $k$th shell of the lattice $\mathbb{Z}^n$ such that the inner product between any two different vectors is at most $t$. Every new code is obtained by scaling a large $(n,k,t,N)$-code which in most cases is associated with a lower bound for $\Sh(n,t,k)$ and possibly augumenting it by (properly scaled) vectors of a different shell. We did not try to augument every code related to a lower bound $\Sh(n,t,k)$ with additional vectors of different shells, it is possible that additional good spherical codes could be constructed this way. It is also possible (but we did not perform these searches systematically) to include vertices associated with different shells of the integer lattice into a weighted graph $G$ and perform maximum weight clique searches in order to find good spherical codes.
\begin{table}[H]\centering
\small\begin{tabular}{|c|c|c||c|}\hline
$n$&Code                     &     Structure                                                    & \cite{EZ}                                               \\\hline\hline
1   &$(6,568,4/5)$         &   $(6,10,8,504)\cup(6,6,4,64)$                          &$(6,558,4/5)$                                         \\\hline 
2   &$(8,1024,5/7)$       &   $(8,7,5,1024)$                                               &$(8,992,5/7)$                                         \\\hline
3   &$(9,612,3/5)$         &    $(9,5,3,612)$                                               &$(9,610,3/5)$                                         \\\hline
4   &$(9,1056,2/3)$       &    $(9,6,4,1056)$                                             &$(9,1008,2/3)$                                       \\\hline
5   &$(9,146,3/7)$         &    $(9,7,3,128)\cup(9,1,0,18)$                           &$(9,130,3/7)$                                         \\\hline
6   &$(9,4450,3/4)$       &    $(9,8,6,4432)\cup(9,1,0,18)$                         &$(9,3586,3/4)$                                       \\\hline
7   &$(9,6404,7/9)$       &    $(9,9,7,5336)\cup(9,5,3,612)\cup(9,3,2,456)$ &$(9,5954,7/9)$                                       \\\hline
8   &$(10,1192,3/5)$     &   $(10,5,3,1192)$                                            &$(10,1190,3/5)$                                     \\\hline
9   &$(10,3584,5/7)$     &   $(10,7,5,3584)$                                             &$(10,3488,5/7)$                                     \\\hline
10 &$(10,13564,7/9)$    &   $(10,9,7,12412)\cup(10,5,3,1152)$                 &$(10,13510,7/9)$                                  \\\hline
11 &$(11,2156,3/5)$     &  $(11,5,3,2156)$                                             &$(11,2154,3/5)$                                     \\\hline
12 &$(11,4664,2/3)$     &  $(11,6,4,4664)$                                             &$(11,4604,2/3)$                                     \\\hline
13 &$(11,80,2/7)$        &  $(11,7,2,80)$                                                 &$(11,79,1/\sqrt{12})$                              \\\hline
14 &$(11,224,3/7)$      &  $(11,7,3,216)\cup(11,1,0,8)$                           &$(11,210,3/7)$                                        \\\hline
15 &$(11,6920,5/7)$     &  $(11,7,5,6720)\cup(11,3,2,200)$                     &$(11,6792,5/7)$                                      \\\hline
16 &$(11,166,3/8)$      &  $(11,8,3,144)\cup(11,1,0,22)$                         &$(11,136,1/3)$ and $(11,210,3/7)$           \\\hline
17 &$(11,13714,3/4)$   &  $(11,8,6,13212)\cup(11,4,3,502)$                   &$(11,12994,3/4)$                                     \\\hline
18 &$(12,2656,3/5)$     &  $(12,5,3,2656)$                                            &$(12,2605,3/5)$                                       \\\hline
19 &$(12,13536,5/7)$   &  $(12,7,5,13504)\cup(12,3,2,32)$                     &$(12,13280,5/7)$                                     \\\hline
20 &$(12,408,4/9)$       &  $(12,9,4,408)$                                              &$(12,396,7/15)$                                      \\\hline
21 &$(12,96,3/11)$       &  $(12,11,3,80)\cup(12,5,1,16)$                        &$(12,95,(13-2\sqrt{2})/7)$                       \\\hline
22 &$(12,1608,4/7)$     &  $(12,14,8,1584)\cup(12,6,3,24)$                     &$(12,1144,5/9)$ and $(12,2605,3/5)$        \\\hline
23 &$(13,3988,3/5)$     &  $(13,5,3,3988)$                                            &$(13,3986,3/5)$                                       \\\hline
24 &$(13,11248,2/3)$   &  $(13,6,4,11248)$                                          &$(13,11176,2/3)$                                      \\\hline
25 &$(13,492,3/7)$      &  $(13,7,3,472)\cup(13,1,0,20)$                        &$(13,416,3/7)$                                          \\\hline
26 &$(13,1978,5/9)$    &  $(13,9,5,1658)\cup(13,6,3,320)$                    &$(13,1490,(1+\sqrt{8})/7)$ and $(13,3986,3/5)$\\\hline
27 &$(13,112,3/11)$    &  $(13,11,3,96)\cup(13,5,1,16)$                        &$(13,104,1/4)$ and $(13,140,19/61)$          \\\hline
28 &$(13,578,5/11)$    &  $(13,11,5,576)\cup(13,1,0,2)$                        &$(13,416,3/7)$ and $(13,1154,1/2)$           \\\hline
29 &$(13,136,4/13)$    &  $(13,13,4,136)$                                            &$(13,104,1/4)$ and $(13,140,19/61)$          \\\hline
30 &$(13,313,5/13)$   &  $(13,13,5,312)\cup(13,1,0,1)$                         &$(13,308,5/13)$                                        \\\hline
31 &$(14,6488,3/5)$   &  $(14,5,3,5464)\cup(14,14,8,1024)$                  &$(14,6460,3/5)$                                        \\\hline
32 &$(14,240,1/3)$     &  $(14,6,2,240)$                                               &$(14,224,1/3)$                                          \\\hline
33 &$(14,19488,2/3)$  &  $(14,6,4,18464)\cup(14,14,8,1024)$                &$(14,19180,2/3)$                                      \\\hline
34 &$(14,600,3/7)$     &  $(14,7,3,600)$                                                &$(14,466,2/5)$ and $(14,896,5/11)$           \\\hline
35 &$(14,55616,5/7)$  &  $(14,7,5,47424)\cup(14,14,10,8192)$              &$(14,52732,5/7)$                                       \\\hline
36 &$(14,136,3/11)$   &  $(14,11,3,112)\cup(14,5,1,24)$                       &$(14,112,1/4)$ and $(14,156,2/7)$              \\\hline
37 &$(14,364,8/21)$   &  $(14,21,8,364)$                                             &$(14,224,1/3)$ and $(14,466,2/5)$              \\\hline
38 &$(15,4800,5/9)$   &  $(15,9,5,4480)\cup(15,6,3,320)$                     &$(15,2734,2/\sqrt{15})$ and $(15,9662,3/5)$\\\hline
39 &$(15,256,7/23)$   &  $(15,23,7,240)\cup(15,15,-1,16)$                    &$(15,240,7/23)$                                          \\\hline
40 &$(16,5768,5/9)$   &  $(16,9,5,5760)\cup(16,7,3,8)$                         &$(16,5408,5/9)$                                         \\\hline
\end{tabular}
\caption{Spherical codes}
\label{tab:t12}
\end{table}

\end{document}